\documentclass[12pt,letter,reqno]{amsart} 

\usepackage{times}

\usepackage[
    top=1in,bottom=1in,
    left=1.5in,right=1.5in
]{geometry}

\usepackage{amssymb,amsfonts,amsthm}
\usepackage{marginnote}
\usepackage{comment}
\usepackage{enumitem}
\usepackage{graphicx}

\usepackage[dvipsnames]{color}

\usepackage[colorlinks=true, pdfstartview=FitV, linkcolor=black, citecolor=blue, urlcolor=blue]{hyperref}

\def\Xint#1{\mathchoice
{\XXint\displaystyle\textstyle{#1}}%
{\XXint\textstyle\scriptstyle{#1}}%
{\XXint\scriptstyle\scriptscriptstyle{#1}}%
{\XXint\scriptscriptstyle%
\scriptscriptstyle{#1}}%
\!\int}
\def\XXint#1#2#3{{\setbox0=\hbox{$#1{#2#3}{%
\int}$ }
\vcenter{\hbox{$#2#3$ }}\kern-.6\wd0}}
\def\barint{\, \Xint -} 
\def\bariint{\barint_{} \kern-.4em \barint}
\def\bariiint{\bariint_{} \kern-.4em \barint}
\renewcommand{\iint}{\int_{}\kern-.34em \int} 
\renewcommand{\iiint}{\iint_{}\kern-.34em \int} 

\DeclareMathAlphabet{\mathcal}{OMS}{cmsy}{m}{n}

\theoremstyle{plain}

\newtheorem{theorem}{Theorem}[section]

\newtheorem{definition}[theorem]{Definition}
\newtheorem{lemma}[theorem]{Lemma}
\newtheorem{example}[theorem]{Example}

\newtheorem{corollary}[theorem]{Corollary}
\newtheorem{proposition}[theorem]{Proposition}

\theoremstyle{definition}
\newtheorem{remark}[theorem]{Remark}

\newcommand{\R}{\mathbb{R}}
\newcommand{\C}{\mathbb{C}}
\newcommand{\N}{\mathbb{N}}
\newcommand{\Z}{\mathbb{Z}}

\newcommand{\be}{\beta}

\newcommand{\e}{\epsilon}
\newcommand{\ga}{{\gamma}}

\newcommand{\si}{\sigma}
\newcommand{\td}{\tilde}

\newcommand{\nb}{\nabla}
\newcommand{\p}{\partial}
\newcommand{\la}{\langle}
\newcommand{\ra}{\rangle}
\newcommand{\les}{\lesssim}

\renewcommand{\:}{\colon}

\let\div\relax
\DeclareMathOperator{\div}{div}

\let\tilde\relas
\newcommand{\tilde}[1]{\widetilde{#1}}

\newcommand{\I}{{\infty}}  

\newcommand{\cmz}[1]{{\color{black} #1}} 

\newcommand{\da}[1]{{\color{black}{#1}}}
\newcommand{\zb}[1]{{\color{black}{#1}}}

\newcommand{\EQ}[1]{\begin{equation}\begin{split} #1 \end{split}\end{equation}}

\numberwithin{equation}{section}
\setlist[enumerate]{leftmargin=*}

\title{Remarks on sparseness and regularity of~Navier--Stokes solutions}  
\author[D. Albritton]{Dallas Albritton}
\address{School of Mathematics, Institute for Advanced Study, 1 Einstein Dr., Princeton, NJ 08540, USA}
\email{dallas.albritton@ias.edu}

\author[Z. Bradshaw]{Zachary Bradshaw}
\address{Department of Mathematics, University of Arkansas, Fayetteville, AR 72701}
\email{zb002@uark.edu}

\begin{document}
\begin{abstract}
The goal of this paper is twofold. First, we give a simple proof that sufficiently sparse Navier--Stokes solutions do not develop singularities. This provides an alternative to the approach of~\cite{Grujic2013}, which is based on analyticity and the `harmonic measure maximum principle'. Second, we analyze the claims in~\cite{algebraicreduction,grujic2019asymptotic} that \emph{a priori} estimates on the sparseness of the vorticity and higher velocity derivatives reduce the `scaling gap' in the regularity problem.
\end{abstract}

\maketitle


\parskip   2pt

\section{Introduction}

\da{
In this paper, we offer an interpretation of the recent program of Gruji{\'c} and coauthors on the relationship between sparseness and regularity of Navier--Stokes solutions, see~\cite{Grujic2013,algebraicreduction,grujic2019asymptotic}. These papers connect sparseness to regularity by way of analyticity and explore the implications of this connection. The most notable implication is the notion of `asymptotic criticality', wherein the
scaling properties of certain sets (`sparseness classes') improve as higher derivatives are considered (we elaborate  {in Section~\ref{subsection:examples}}).
As the number of derivatives  {tends to infinity}, 
 {the} \emph{a priori}  {controlled}
supercritical quantities associated with the aforementioned  {sparseness classes seems to} approach criticality~\cite{grujic2019asymptotic}. In this sense, the scaling gap is  {said to be} closed \textit{within the sparseness framework}.   {This prospect is intriguing, as} the scaling gap is a fundamental barrier to establishing global regularity for the Navier--Stokes equations.

The aims of this paper are twofold:  First, we offer new approaches to the  estimates introduced in~\cite{Grujic2013} which are the foundation of the later works~\cite{algebraicreduction,grujic2019asymptotic}. Second, we analyze the  {sparseness classes} which  {seem to} bridge the scaling gap in  \cite{algebraicreduction,grujic2019asymptotic} by way of concrete examples. In the first direction, we reframe the analyticity-based argument of~\cite{Grujic2013} in terms of methods more widely known in the fluids community. 
Furthermore, we introduce a {frequency} version of sparseness which has the benefit of involving fewer parameters. We hope this simplified setting will make the higher derivative estimates in~\cite{grujic2019asymptotic} more transparent. In the second direction,  we  {seek to}
understand exactly what is happening when the sparseness classes of \cite{grujic2019asymptotic}  {seem to} approach criticality.   In particular,  we show for a broad class of concrete examples  {of vector fields} 
that the framework introduced in \cite{grujic2019asymptotic} does not rule out singularities beyond those ruled out by membership in $L^\infty_t L^2_x$.  Our examples suggest that membership in the  {sparseness classes} 
of~\cite{grujic2019asymptotic} may not be enough to meaningfully bridge the scaling gap insofar as regularity is concerned.  
 
}

\subsection{Introduction to sparseness}

\da{We give an introductory discussion of the notion of sparseness and discuss its connection to the regularity problem.} A measurable function $f : \R^d \to \R$ satisfying
\begin{equation}
    \label{eq:naivesparseness}
    \left| \Big\lbrace |f| > \frac{1}{2} \| f \|_{L^\infty} \Big\rbrace \cap B_{\ell}(x_0) \right|  \leq \varepsilon |B_\ell|
\end{equation}
for all $x_0 \in \R^d$ is said to be \emph{$\varepsilon$-sparse at scale $\ell$}.  \da{We will soon define a more flexible notion of sparseness and only consider \eqref{eq:naivesparseness} briefly for illustrative purposes.} Heuristically, sparseness means that the set where the function $f$ is substantial is relatively small in measure.  
In particular, sparseness ensures that $f$ concentrates on a small set. For example,
\begin{equation}
    f = \frac{1}{2} + \frac{1}{2} \mathbf{1}_{(-\frac{\varepsilon}{2},\frac{\varepsilon}{2})}
\end{equation}
is $\varepsilon$-sparse at scale $\ell = 1$ in dimension $d=1$. 

Sparseness is a \emph{non-linear} requirement. In particular, it is not amenable to a vector space structure, although it is invariant under multiplying $f$ by a prefactor. It is also invariant under translation and rotation, whereas dilation simply dilates $\ell$.

The original motivation for the PDE analysis of sparseness of  Navier--Stokes solutions comes from the phenomenology of turbulence, see \cite{Grujic2013}. Loosely speaking, turbulent flows are observed to be concentrated on vortex filaments, which are observed to be sparse.  


A key feature of the above class is that solutions of the heat equation with sparse initial data $f$ rapidly decrease in norm after time $T = C \ell^2$:
\begin{equation}
    \label{eq:heatdropsinnormintro}
    \| e^{t \Delta} f \|_{L^\infty} \leq \frac{3}{4} \| f \|_{L^\infty}, \quad t \geq T.
\end{equation}
Sparseness requires that $f$ has substantial variation or `large gradient' over scale $\ell$, and the heat evolution smooths this variation by decreasing the function. Since the heat kernel is roughly concentrated on a ball of radius $C \sqrt{t}$, the variation on scale $\ell$ is not seen until time $T$. These heuristics essentially comprise the proof we write in Lemma~\ref{lemma.heat}. It is also possible to take a Fourier-analytic perspective: a substantial amount of the function is concentrated in frequencies $\geq C \ell^{-1}$, see Definition~\ref{def:frequencysparseness} and Lemma~\ref{lemma.heat.frequency}.

Since $L^\infty$ is a subcritical space for the Navier--Stokes equations, it is reasonable to expect that, if the initial velocity $u_0$ is sufficiently sparse, then the Navier--Stokes solution also undergoes a rapid decrease in norm, analogous to~\eqref{eq:heatdropsinnormintro}, at least while the non-linearity is perturbative. In summary, sufficient sparseness should allow the Navier--Stokes solution to be continued.  This is essentially the proof of Theorem~\ref{thm:regcrit} below.

We emphasize that the simple heuristics above are robust enough to be translated into a simple proof.

The route taken by Gruji{\'c}~\cite{Grujic2013} to the regularity criterion introduced above is different. The rapid decrease in~\eqref{eq:heatdropsinnormintro} is observed by showing that the solution is spatially analytic with radius $C \sqrt{t}$. Hence, it can be extended to an analytic function in a strip in $\C^d$, on which its real and imaginary parts are harmonic and thus satisfy a certain `harmonic measure majorization principle'. Extending this argument to the non-linear setting is done by demonstrating analytic smoothing of Navier-Stokes solutions, which can be technical, \da{especially at the level of higher derivative estimates in \cite{grujic2019asymptotic}. We hope that this alternative approach or the frequency
version that we introduce in Section 2 will help streamline these estimates.  }

\da{It is well known that small scale activity is essential to singularity formation. This can be quantified in terms of an } 
\emph{a priori} lower bound on the sparseness of a singular Navier--Stokes solution. We explain the idea at the level of functions $f \in L^1 \cap L^\infty$, though any two Lebesgue spaces will do. In terms of dimensional analysis, we have
\begin{equation}
	\big[ \| f \|_{L^1} \big] = ML^d, \quad \big[ \| f \|_{L^\infty} \big] = M.
\end{equation}
The unique way, up to a prefactor, to form a length scale from these two quantities is
\begin{equation}
	\ell_0 = \left( \frac{ \| f \|_{L^1}}{ \| f \|_{L^\infty} } \right)^{\frac{1}{d}}.
\end{equation}
 Heuristically, $\ell_0$ is a length scale at which the two quantities are in balance. We make sense of this with Chebyshev's inequality:
\begin{equation}
    \label{eq:chebyshev}
	\lambda \left| \{ |f| > \lambda \} \right| \leq \| f \|_{L^1}, \quad \lambda > 0.
\end{equation}
By choosing $\lambda = \| f \|_{L^\infty}/2$, we conclude
\begin{equation}
	\left| \Big\lbrace |f| > \frac{1}{2} \| f \|_{L^\infty} \Big\rbrace \right| \leq 2 \frac{\| f \|_{L^1}}{\| f \|_{L^\infty}} = 2 \ell_0^d,
\end{equation}
an upper bound on the volume where the function is substantial.\footnote{In the opposite direction, one obtains an upper bound on the sparseness (a lower bound on the volume where the function is substantial) by controlling $\| f \|_{L^\infty}$ from below and $\| \nabla f \|_{L^\infty}$ from above.} In particular, we have that the solution is sparse on length scales $\ell = a \ell_0$ with $a \gg 1$:
\begin{equation}
	\left| \Big\lbrace |f| > \frac{1}{2} \| f \|_{L^\infty} \Big\rbrace \cap B_\ell(x_0) \right| \leq C_d a^{-d} |B_\ell|.
\end{equation}
In the context of singular Navier--Stokes solutions, whose $L^2$ norm is controlled, the above reasoning demonstrates that the length scale of sparseness is guaranteed to shrink with a certain rate as $t \to T^*$, see Theorem~\ref{thm:apriori}.


We now formulate the main results of this paper. First, we define

\begin{definition}[$L^p$-sparseness]
\label{def:lpsparseness}
Let $1 \leq p \leq \infty$, $ \varepsilon,  \be\in (0,1)$, and $ \ell>0$. A vector field $u_0 \in L^p(\R^d)$ is \emph{$(\varepsilon,\be,\ell)$-sparse in $L^p$} if there exists a measurable set $S$ such that
\begin{equation}
    \|u_0\|_{L^p (S^c)} <    \beta  \|u_0\|_{L^p}
\end{equation}
and 
\begin{equation}
\sup_{x_0\in \R^d} \frac {|S\cap B_{  \ell }(x_0) |} {|B_{ \ell}(x_0) |} \leq \varepsilon.    
\end{equation}
\end{definition} 
\da{This definition can be extended to spaces other than the Lebesgue spaces. It generalizes the volumetric  notion of sparseness used in  \cite{algebraicreduction,grujic2019asymptotic} but is different than the one-dimensional version in \cite{Grujic2013}.}

Let $d \geq 3$ and $p \in (d,+\infty]$. We define the \emph{guaranteed existence time} $\bar{T}(d,p,\| u_0 \|_{L^p}) > 0$ by
\begin{equation}
    \label{eq:guaranteedexistencetimedef}
    \| u_0 \|_{L^p} \bar{T}^{\frac{1}{2} (1 - \frac{d}{p})} = c_p,
\end{equation}
where $c_p \in (0,1]$ is a small constant.
For each $u_0 \in L^p_\sigma$, there exists a unique mild Navier--Stokes solution $u \in C((0,\bar{T}];L^p)$ which furthermore satisfies\footnote{When $p < +\infty$, we have also $C([0,\bar{T}];L^p)$.}
\begin{equation}
    \sup_{t \in (0,\bar{T})} \| u(\cdot,t) \|_{L^p} \leq 2 \| u_0 \|_{L^p}.
\end{equation}
\da{Classical bilinear estimates (see, for example,~\cite[Chapter 5]{Tsaibook}) imply that
for $t \in (0,\bar{T})$, we have
\begin{equation}\label{ineq.Duhamel}
\bigg\|  \int_0^t e^{(t-s)\Delta}\mathbb P \nb \cdot (u\otimes u)\,ds\bigg\|_{L^p}
\leq C_1 t^{\frac{1}{2}(1-\frac{d}{p})}   \|u_0\|_{L^p}^2,
\end{equation}
for a constant \da{$C_1 \geq (2c_p)^{-1}$} that will play a role in what follows.}

\da{The following theorem is thematically similar to the main result in \cite{Grujic2013}, which is the foundation for later work \cite{algebraicreduction,grujic2019asymptotic}.}
\begin{theorem}[Regularity criterion]
    \label{thm:regcrit}
    Let $u_0\in L^p_\si(\R^d)$ and $T_0 > \bar{T}(d,p,\| u_0 \|_{L^p})$ (otherwise, the conclusion is trivial).
    Define $\gamma$ according to 
\begin{equation}
    \| u_0 \|_{L^p} T_0^{\frac{1}{2}(1-\frac{d}{p})} \gamma = c_p,
\end{equation}
and in particular, $\gamma \in (0,1).$  \da{Define $T_1$
 according to
\begin{equation}
    \label{eq:Tdef}
    \| u_0 \|_{L^p} T_1^{\frac{1}{2} (1 - \frac{d}{p})} = \frac{\gamma}{2C_1},
\end{equation}
\da{which when compared to~\eqref{eq:guaranteedexistencetimedef} ensures that $T<\bar T$.}}
There exists $C_0 \geq 1$, depending only on $p \in (d,+\infty]$ and the dimension $d$, such that, if $u_0$ is $(\varepsilon,\be,\ell)$-sparse in $L^p$, where  \da{$\ell = \sqrt{T_1} \bar{\ell}$} and the dimensionless parameters satisfy
\begin{equation}
    \label{eq:introreqsonelletc}
  \bar{\ell}^2 \geq C_0 {\log (C_0/\gamma)}; \quad  \da{\be}  = C_0^{-1} \gamma;  \quad  \da{\varepsilon}^{1-\frac{1}{p}} \leq C_0^{-1} \gamma / \bar{\ell}^d
\end{equation}
then the unique strong solution with initial data $u_0$ extends past time $T_0$. 
\end{theorem}

\da{
The above theorem may be regarded as a generalization of Gruji{\'c}'s  work~\cite{Grujic2013} to the $L^p$ scale.\footnote{We mention two differences. First, the paper~\cite{Grujic2013} is written in terms of so-called `linear sparseness', although the further developments~\cite{algebraicreduction,grujic2019asymptotic} focused on `volumetric sparseness' in $L^\infty$, which is akin to Definition~\ref{def:lpsparseness} and which we focus on here. Second, the sparsity in the regularity criterion in~\cite[Theorem 4.1]{Grujic2013} and subsequent papers is imposed with a `time lag', so their hypotheses are not the same as ours.} However, the main novelty is that the proof avoids analyticity  {and the harmonic measure maximum principle} in favor of methods widely used in the fluids literature. In principle, it can be adapted to  generalize~\cite{algebraicreduction,grujic2019asymptotic}.
}

\da{We obtain as  a corollary a blow-up criterion which may be regarded as {saying that a singular solution can only exhibit} `slow' concentration at small spatial scales. This highlights the usefulness of sparseness as a tool to quantify blow-up.}

\begin{corollary}[Blow-up criterion]
    \label{cor:blowupcrit}
Suppose that $u$ is a strong Navier--Stokes solution with initial data $u_0 \in L^p_\sigma$. Suppose that the maximal time of existence $T^*(u_0)$ is finite. Let $\gamma_p(t)$ satisfy
\begin{equation}
     \| u(\cdot,t) \|_{L^p} (T^* - t)^{\frac{1}{2} (1-\frac{d}{p})} \gamma_p(t) = c_p
\end{equation}
\da{and $T_1(t)$ be defined by~\eqref{eq:Tdef} with $\gamma_p(t)$ replacing $\gamma$}.
Then, for all $t \in (0,T^*)$, the solution $u(\cdot,t)$ \emph{fails to be} $(\varepsilon,\beta,\ell)$-sparse in $L^p$ with $\ell = \da{ \sqrt{T_1(t)} \bar{\ell}}$ and dimensionless parameters satisfying~\eqref{eq:introreqsonelletc} with $\gamma = \gamma_p(t)$.


\end{corollary}

 \da{ Finally, we include a theorem which demonstrates how sparseness necessarily occurs just prior to a blow-up time.}  The case $p=\infty$ is discussed in~\cite{Grujic2013}.

\begin{theorem}[\emph{A priori} sparseness]
    \label{thm:apriori}
    Suppose that $u$ is a strong Navier--Stokes solution with initial data $u_0 \in L^2_\sigma \cap L^p_\sigma$. Let
    \begin{equation}
    \label{eq:ellnotdef}
    \ell_0(t) = \left( \frac{\| u(\cdot,t) \|_{L^2}}{\| u(\cdot,t) \|_{L^p}} \right)^{\mu_p}, \quad \frac{1}{\mu_p} = d \left( \frac{1}{2} - \frac{1}{p} \right).
\end{equation}
    Then, for all $\varepsilon, \beta \in (0,1)$, there exists $a \geq 1$ such that $u(\cdot,t)$ is $(\varepsilon,\beta,\ell(t))$-sparse in $L^p$ for all $t \in (0,T^*(u_0))$, where $\ell(t) = a \ell_0(t)$. In particular, if the maximal time of existence $T^*(u_0)$ is finite, then
    \begin{equation}
        \label{eq:mustconcentrate}
    \ell_0^{\frac{1}{\mu_p}} \leq c_p^{-1} \| u_0 \|_{L^2} (T^* - t)^{\frac{1}{2}(1-\frac{d}{p})} \to 0 \da{\text{ as } t \to T^*_-}.
    \end{equation}
\end{theorem}

\da{Together, Theorem~\ref{thm:apriori} and Corollary~\ref{cor:blowupcrit} assert that, just prior to a blow-up time, activity must concentrate on small scales,  see~\eqref{eq:mustconcentrate}, but cannot concentrate too quickly.}    \da{
In Section~\ref{sec:frequencysparseness}, we obtain analogues of the above theorems for the frequency sparseness we introduce in Definition~\ref{def:frequencysparseness}. Frequency sparseness has the added benefit that it involves fewer parameters, which simplifies the proofs and makes it compelling for applications.
}

\subsection{Examples concerning `asymptotic criticality'}
\label{subsection:examples}

Our second goal, which is perhaps more important, is to elucidate the so-called `asymptotic criticality' 
introduced in~\cite{grujic2019asymptotic}.  

In~\cite{algebraicreduction}, it was proposed that the `scaling gap' between the energy class, where the \emph{a priori} estimates live, and the critical spaces, where the regularity criteria live, could be reduced by an `algebraic factor' by analyzing the \emph{vorticity} sparseness. The natural extension to sparseness of $\nabla^k u$ was implemented in~\cite{grujic2019asymptotic}.  The authors develop analogues of Theorems~\ref{thm:regcrit} and~\ref{thm:apriori} and interpret them as closing the scaling gap, within the sparseness framework, asymptotically as $k \to +\infty$.

We offer a different interpretation of~\cite{algebraicreduction,grujic2019asymptotic} which suggests that, in a certain reasonable sense, the `scaling gap' is not improved beyond the energy class. This is discussed below and in the examples in Section~\ref{sec:examples}. 

To understand the interpretation in~\cite{algebraicreduction,grujic2019asymptotic}, we codify the sparseness classes introduced therein. The class $Z_\alpha^{(k)}$ is defined by the requirement that $\nabla^k f$ is sparse, with some value of parameters,  with the sparseness scale $\ell$ assumed to satisfy $\ell \approx \| \nabla^k f \|_{L^\infty}^{-\alpha}$ (bounded above and below, up to multiplicative constants). $Z_\alpha^{(k)}$ implicitly depends on these parameters and constants.

To understand better the class $Z^{(k)}_\alpha$, we identify its scaling symmetry. Recall that, if $f$ is sparse at scale $\ell_{\rm old}$, then $\lambda^\beta f(\lambda x)$ is sparse at scale $\ell_{\rm new} = \ell_{\rm old}/\lambda$ for any $\beta \in \R$.  A \da{calculation} shows that $Z^{(k)}_{\alpha}$ is invariant under the scaling symmetry
\begin{equation}
	f_\lambda(x) = \lambda^{\frac{1}{\alpha} - k} f(\lambda x),
\end{equation}
which makes $\ell_{\rm new} \sim \| \nabla^k f_\lambda \|_{L^\infty}^{-\alpha}$. (The choice $\beta = 1/\alpha - k$ is the only choice which does so.) Therefore, $Z_{\alpha}^{(k)}$ has the same homogeneity as $\dot W^{k-\frac{1}{\alpha},\infty}$, and we write \emph{informally}
\begin{equation}
    \label{eq:thehomogeneity}
    Z_{\alpha}^{(k)} \sim W^{k-\frac{1}{\alpha},\infty}.
\end{equation}
This `identification' should not be taken seriously and is easily abused, as we see below. This is, perhaps, because the classes $Z_{\alpha}^{(k)}$ should not be conflated with what we conceive of as function spaces, such as Lebesgue and Sobolev spaces, which are actually Banach spaces.

When $k = 0$ and $d=3$, we have
\begin{equation}
    L^{p,\infty} \subset Z^{(0)}_\alpha \text{ with } \alpha = \frac{p}{3},
\end{equation}
and indeed $L^{p,\infty}$ has the same homogeneity as $Z^{(0)}_\alpha$. The notation $\subset$ is only set inclusion: $Z^{(0)}_\alpha$ is not a normed vector space. The informal relationship $L^{3\alpha,\infty} \sim Z^{(0)}_\alpha$ is used systematically in~\cite{algebraicreduction}. Hence, the discovery in~\cite{algebraicreduction} that
\begin{equation}
\omega \in Z^{(0)}_{\frac{2}{5}} \sim L^{\frac{6}{5},\infty}
\end{equation}
`uniformly' up until a putative finite-time blow-up, was interpreted as a reduction in the scaling gap (recall that $\omega \in L^\infty_t L^1_x$ is at the level of the energy class). In terms of homogeneity in~\eqref{eq:thehomogeneity}, however, we have that
\begin{equation}
    u \in Z^{(1)}_{\frac{2}{5}} \sim W^{-\frac{3}{2},\infty},
\end{equation}
which has the same homogeneity as $L^2$. Similarly, in~\cite{grujic2019asymptotic}, the \emph{a priori} sparseness exponent $\bar{\alpha}_k$ corresponding to $\nabla^k u$ is
\begin{equation}
	\bar{\alpha}_k = \frac{1}{k+\frac{d}{2}}.
\end{equation}
Notice that
\begin{equation}
    Z^{(k)}_{\bar{\alpha}_k} \sim \dot W^{-\frac{d}{2},\infty},
\end{equation}
which has the same homogeneity as $L^2$.

While the above interpretation is suggestive, still, the relationship~\eqref{eq:thehomogeneity} is informal and should not be taken seriously. Rather, to better demonstrate that the scaling gap is not reduced, we analyze the sparseness of $\nabla^k u$ in a class of concrete 
blow-up scenarios in Section~\ref{sec:examples}. \da{These scenarios are motivated by the blow-ups known to occur in related nonlinear PDEs.} Among these scenarios, the $Z^{(k)}_{\bar{\alpha}_k}$ classes corresponding to the \emph{a priori} sparseness in~\cite{grujic2019asymptotic} do not exclude more blow-ups than finite kinetic energy  $u \in L^\infty_t L^2_x$ already does.



\subsection{Existing literature}

\da{
The existing literature concerning Navier-Stokes regularity is massive, and we mention only a few threads connected to~\cite{Grujic2013,BradshawFreqLocalized,grujic2019asymptotic} and the present work. One way to understand this program is as connecting rates of spatial concentration to potential singularity formation. In particular, sparseness is closely related to the concentration phenomena studied
in~\cite{LiOzawaWang} and~\cite{localizedsmoothingBarkerPrange}.\footnote{{See also~\cite{KangMiuraTsai} where the $p\geq 3$ condition in~\cite{localizedsmoothingBarkerPrange} is reduced to $2$ by passing to a Morrey scale.}} The results therein are dedicated to lower bounds on hypothetical singular solutions in balls whose radii are bounded above, and they have applications to quantitative $L^\infty_t L^3_x$ blow-up criteria~\cite{quantitativeregularitybarkerprange}. The criteria~\cite{grujic2021regularity} of Gruji{\'c} and Xu may be considered as in the vein of~\cite{LiOzawaWang} but through the lens of sparseness and analyticity. We also observe that~\cite{LiOzawaWang} contains a regularity criterion in the spirit of Theorem~\ref{thm:regcrit}; they ask that $u_0$ is supported in high frequencies $|\xi| \gg \| u_0 \|_{L^\infty}$. A closely related notion to our frequency sparseness is that of `dissipation wavenumber', explored in~\cite{cheskidovshvydkoyunified} and subsequent works; this is in turn related to the `bubbles of concentration' in Tao's quantitative $L^\infty_t L^3_x$ criterion~\cite{tao2019quantitative}. Furthermore, sparseness and concentration are directly related to intermittency, and in particular, to the analytical approach to intermittency introduced in~\cite{CheskidovShvydkoyanalytical} and the concept of characteristic speeds and active regions therein.

\cmz{The original result of Gruji\'c~\cite{Grujic2013} can be viewed as a geometric regularity criterion. Such results have a rich history within the analysis of fluid equations. See, for example, the pioneering papers~\cite{CoFe,ChCh}, which, respectively, initiated research into regularity based on the alignment of the vorticity and membership of reduced components in critical classes, as well as the references in the survey papers~\cite{BGZ,Miller}.}

We mention two further works in the program of~\cite{grujic2019asymptotic}. In~\cite{grujic2020time}, Gruji{\'c} and Xu extend the tools developed in \cite{grujic2019asymptotic} to the hyper-dissipative Navier--Stokes equations to analyze certain geometric blow-up scenarios. 
Their program has also recently informed the computational study~\cite{rafner2021geometry} on small length scales in turbulence.
}

\section{Preliminaries}



Let $G : \R^d \to \R$ be a Schwartz function and $G_t$ be the convolution operator
\begin{equation}
    G_t u_0 := t^{-\frac{d}{2}} G(\cdot/\sqrt{t}) \ast u_0.
\end{equation}
when $t > 0$. We have in mind that $G = (4\pi)^{-d/2} e^{-|x|^2/4}$ and $G_t$ is the heat semigroup, though it will be convenient to allow $G$ to be general.

\begin{lemma}\label{lemma.heat}
 Let $p \in (1,\infty]$, $\gamma \in (0,1)$, and $t>0$ \da{be fixed}. Let $u_0\in L^p(\R^d)$ be a vector field. Suppose that $u_0$ is $(\varepsilon,\be,\bar \ell \sqrt {t})$-sparse, where the dimensionless parameters $\varepsilon, \be \in (0,1)$ and $\bar\ell > 0$ satisfy
 \begin{equation}
    \label{eq:therequirementsforsparsitytouselater}
    \bar \ell \geq f(\gamma);\quad  \da{\beta} \leq \| G \|_{L^1}^{-1} \gamma / 3;\quad  \da{\varepsilon^{1-\frac{1}{p}}} \leq C_0^{-1} \| G \|_{L^\infty}^{-1} \gamma / \bar{\ell}^d
\end{equation}
where $f$ depends on $G$ and satisfies $f(\gamma) \to +\infty$ as $\gamma \to 0^+$, and
$C_0 > 1$ is an absolute constant depending only on the dimension.  Then \begin{equation}\label{ineq.caloric.decay}
\|  G_t u_0\|_{L^p} \leq \gamma \|u_0\|_{L^p}.    
\end{equation}
When $G_t = e^{t \Delta}$, the above requirement on $\bar{\ell}$ can be made more explicit:
 \begin{equation}
    \label{eq:therequirementsforlesparsity}
    \bar \ell^2 \geq C_0 {\ln(C_0/\gamma)}.
\end{equation}
\end{lemma}

\da{Notice that the length scale of sparseness $\bar{\ell} \sqrt{t}$ in Lemma~\ref{lemma.heat} depends on $t$.}

The case $p=1$ fails since, for non-negative initial data, the heat equation preserves the $L^1$ norm, regardless of sparsity. 

\begin{proof} We write only the proof for $p \in (1,\infty)$. The endpoint case $p=\I$ is identical except that the outer norm in~\eqref{ineq.B} is replaced by a $\sup$.  

 \da{
Upon rescaling, we need only consider $t = 1$, a fact we now justify. Let $u_0(x)=\td u_0(\sqrt t x)$. Then a simple computation shows
\[
(G_t * \td u_0) (x)=  (G_1 * u_0) (x/\sqrt t).
\]
Using this, we have 
\[
\| G_t * \td u_0 \|_{L^p} =  \sqrt t^{\frac d p} \| G_1 *u_0\|_{L^p}.
\]
By the definition of sparseness, if $\td u_0$ is $(\varepsilon,\be,\bar\ell \sqrt t)$-sparse, 
then $u_0$ is $(\varepsilon,\be,\bar\ell  )$-sparse. 
Hence, assuming the result for $t=1$, we have for a given vector field $\td u_0$ that
\[
\| G_t * \td u_0 \|_{L^p} = \sqrt t^{\frac d p} \| G_1 \ast u_0\|_{L^p} \leq \sqrt t^{\frac d p} \ga \| u_0\|_{L^p} = \ga \| \td u_0 \|_{L^p} \, .
\]

}

 \da{
Assuming $t=1$,} we expand the operator $G_1$ into three parts:
\[
G_1 u_0= {\rm I}^{\rm far}+{\rm II}^{\rm near}_{S}+{\rm II}^{\rm near}_{S^c},
\]
where
\begin{equation}
\begin{aligned}
    {\rm I}^{\rm far}(x)&=\int  G(x-y) \chi_{B_{\bar \ell}^c}(x-y) u_0 (y)  \,dy
\\{\rm II}^{\rm near}_{S}(x)&= \int  G(x-y)  \chi_{B_{\bar \ell}}(x-y) u_0 (y) \chi_{S}(y)\,dy
\\{\rm II}^{\rm near}_{S^c}(x)&=\int  G(x-y)  \chi_{B_{\bar \ell}}(x-y) u_0 (y) \chi_{S^c}(y)\,dy. 
\end{aligned}
\end{equation}
We will first fix $\bar \ell$ and then $\varepsilon$ and $\be$.

By Young's convolution inequality, the first term, ${\rm I}^{\rm far}$, satisfies
\begin{equation}
    \|{\rm I}^{\rm far}\|_{L^p} \leq \| u_0\|_{L^p} \| G \|_{L^1(B_{\bar \ell}^c)} 
\les_{G,k} \la \bar{\ell} \ra^{-k} \| u_0\|_{L^p}
\end{equation}
for all $k \geq 0$, since $G$ is a Schwartz function. By choosing $\bar \ell$ sufficiently large,  we can therefore ensure
\[
\| {\rm I}^{\rm far}\|_{L^p} \leq \frac \gamma 3 \|u_0\|_{L^p}.
\]
When $G$ corresponds to the heat kernel, we have $\| G \|_{L^1(B_{\bar \ell}^c)} \leq C e^{-\bar{\ell}^2/8}$, which gives us the requirement on $\bar{\ell}$ in~\eqref{eq:therequirementsforlesparsity}.

The second term, ${\rm II}^{\rm near}_{S}$, is the most interesting:
\begin{equation}
\label{ineq.B}
    \begin{aligned}
    \| {\rm II}^{\rm near}_{S} \|_{L^p}&=  \bigg(   \int   \bigg|\int G(x-y) u_0(y) \chi_{B_{{\bar \ell}}(x)}(y)\chi_S(y)\,dy \,\bigg|^p \,dx           \bigg)^{\frac 1 p}
\\&\leq \| G \|_{L^\infty} \bigg(    \int  \bigg( \| u_0 \|_{L^p(B_{\bar \ell}(x))} \, |S\cap B_{{\bar \ell}} (x)|^{\frac{1}{p'}} \bigg)^p \,dx           \bigg)^{\frac 1 p}
\\&\lesssim \| G \|_{L^\infty} \sup_{x\in \R^d }|S\cap B_{{\bar \ell}}(x)|^{\frac{1}{p'}}   ({\bar \ell})^{\frac{d}{p}} \|u_0\|_{L^p},
    \end{aligned}
\end{equation}
where the suppressed constant depends on the volume of the unit ball and $|\cdot|$ refers to Lebesgue measure.
By the sparseness assumption, this becomes 
\EQ{
\|{\rm II}^{\rm near}_{S}\|_{L^p}& \les \| G \|_{L^\infty}  \da{\varepsilon^{\frac{1}{p'}}} {\bar \ell}^d      \|u_0\|_{L^p}.
}
We choose  \da{$\varepsilon$} small to ensure that 
\[
\| {\rm II}^{\rm near}_{S}\|_{L^p} \leq \frac \gamma 3  \|u_0\|_{L^p}.
\]
Finally, for ${\rm II}^{\rm near}_{S^c}$, by Young's convolution inequality, we have \begin{equation}
\|{\rm II}^{\rm near}_{S^c}\|_{L^p}\leq \| u_0 \chi_{S^c}\|_{L^p} \| G \|_{L^1} \leq  \da{\be}  \| G \|_{L^1} \|u_0\|_{L^p}.
\end{equation}
By taking $ \da{\be} \leq \| G \|_{L^1}^{-1} \gamma/3$, we are done. \end{proof}


$L^p$-sparseness is formulated in physical space, but a similar conclusion in Fourier space follows if the initial data is supported on sufficiently high Littlewood-Paley frequencies. See~\cite[Chapter 2]{BahouriCheminDanchin} for a review of Littlewood-Paley theory. Let $\Delta_{\geq J} = \sum_{j\geq J}\dot \Delta_j$ and $  \Delta_{<J}=\sum_{j<J} \dot \Delta_j  $. Here, $2^J$ is a frequency, and $2^{-J}$ is a length scale. It is not essential that $J$ is an integer.

\begin{definition}[$L^p$-sparseness in frequency]
\label{def:frequencysparseness}
Let $\be \in (0,1)$ and $J\in \R$. Then a vector field $u_0 \in L^p$ is \emph{$(\be,J)$-sparse in frequency in $L^p$} if 
\begin{equation}
\| \Delta_{<J}u_0\|_{L^p}\leq \be \|u_0\|_{L^p}.    
\end{equation}
\end{definition}

\da{This notion of sparseness actually encompasses the spatial version, at least within a certain parameter range, as we demonstrate in Lemma~\ref{lemma:sparial.vs.frequency}, while preserving the caloric decay property \eqref{ineq.caloric.decay}, see Lemma \ref{lemma.heat.frequency}. }


\begin{lemma}[Spatial vs. frequency sparseness]\label{lemma:sparial.vs.frequency} Let $1 \leq p \leq \infty$, $\gamma \in (0,1)$, $J \in \R$, and $\sqrt{t} = 2^{-J}$. Let $u_0 \in L^p(\R^d)$ be a vector field. Suppose that $u_0$ is $(\varepsilon,\beta,\bar{\ell} \sqrt{t} )$-sparse with dimensionless parameters  $\varepsilon, \be \in (0,1)$ and $\bar\ell > 0$ satisfying~\eqref{eq:therequirementsforsparsitytouselater}. Then $u_0$ is $(\gamma,J)$-sparse in frequency.
\end{lemma}
\begin{proof}
Let $G$ be the Schwartz function associated to the convolution operator $\Delta_{< 1}$. Then the proof is a direct application of Lemma~\ref{lemma.heat}.
\end{proof}

\begin{lemma}\label{lemma.heat.frequency}Fix $1\leq p\leq \I$, $t>0$ and $\ga>0$. Let $u_0\in L^p$. There exists $J\in \Z$ satisfying $2^{J}\sim \ga^{-1} t^{-1/2}$ and $\be=\ga/2$ so that, if $u_0$ is $(\be,J)$-sparse in frequency, then  
\begin{equation}
\| e^{t\Delta} u_0 \|_{L^{p}} \leq \ga \|u_0\|_{L^p}.    
\end{equation}
\end{lemma}

\begin{proof} 
 By \cite[Lemma 2.4]{BahouriCheminDanchin}, for any $1\leq p\leq \infty$,
 \begin{equation}
\| e^{t\Delta} \dot \Delta_{j} u_0\|_{L^p}\leq Ce^{-c t 2^{2j}} \|u_0\|_{L^p}.     
 \end{equation}
Hence,
\begin{equation}
    \| e^{t\Delta}  \Delta_{\geq J} u_0\|_{L^p} 
\lesssim  \sum_{j\geq J} e^{-c t 2^{2j}} \|u_0\|_{L^p}
 \lesssim \bigg(\sum_{j\geq J} \frac 1 {t2^{2j}} \bigg) \|u_0\|_{L^p}\lesssim \frac 1 {t 2^{2J}} \|u_0\|_{L^p} < \frac\ga2 \|u_0\|_{L^p},
\end{equation}
provided $2^{J}\sim \ga^{-1} t^{-1/2}$.

On the other hand, by assumption
\begin{equation}
    \| e^{t\Delta }\Delta_{< J } u_0  \|_{L^p} \leq  \| \Delta_{< J } u_0 \|_{L^p} < \be \|u_0\|_{L^p},
\end{equation}
and we are done provided $\be = \ga/2$. \end{proof} 

\begin{remark}
One could also give a frequency sparseness definition in homogeneous Besov spaces. This is essentially what is done in~\cite{BradshawFreqLocalized}.
\end{remark}

\section{Main results}

\subsection{Sparseness in physical space}

In this section, we prove Theorem~\ref{thm:regcrit}, Corollary~\ref{cor:blowupcrit}, and Theorem~\ref{thm:apriori}. First, we adapt Lemma~\ref{lemma.heat} to the nonlinear setting:

\begin{proposition}\label{prop.decreasing}
Let $d < p \leq \infty$ and $u_0 \in L^p_\sigma(\R^d)$. For any $\gamma \in (0,1)$, \da{define $T_1 \in (0,\bar T)$ according to
\eqref{eq:Tdef}. Then $T_1$ satisfies} the following properties: First, the strong $L^p$ solution $u$ of the Navier-Stokes equations with initial data $u_0$ exists on $\R^d \times (0,\da{T_1}]$. Moreover, for all $t\in (0,\da{T_1}]$, if $u_0$ is $(\varepsilon,\be,\ell)$-sparse, where $\ell = \sqrt{t} \bar{\ell}$ and
\begin{equation}
   \da{\be}  = \frac \gamma 6;\quad \da{\varepsilon}^{1-\frac{1}{p}} \leq (2C_0)^{-1} \gamma / \bar{\ell}^d; \quad , \bar \ell^2 \geq C_0 {\ln(2 C_0/\gamma)},
\end{equation}
then
\begin{equation}
    \| u(\cdot,t)\|_{L^p} \leq \gamma \|u_0\|_{L^p}. 
\end{equation}
\end{proposition}


\begin{proof}

 \da{ From \eqref{ineq.Duhamel} and the definition of $T_1$,
for all $t \in (0,T_1]$, we have }
\begin{equation}
    \label{eq:mildsolformulation1}
\bigg\|  \int_0^t e^{(t-s)\Delta}\mathbb P \nb \cdot (u\otimes u)\,ds\bigg\|_{L^p}
\leq  \frac \gamma 2 \|u_0\|_{L^p}.
\end{equation}
Next,  we apply Lemma~\ref{lemma.heat} (with $\gamma/2$ instead of $\gamma$ therein) to obtain
\begin{equation}
\label{eq:mildsolformulation2}
\| e^{t\Delta}u_0\|_{L^p}\leq \frac \gamma 2 \|u_0\|_{L^p}
\end{equation}
under the requirements on $(\varepsilon,\beta,\ell)$ in Lemma~\ref{lemma.heat}.
Finally, we combine~\eqref{eq:mildsolformulation1}, ~\eqref{eq:mildsolformulation2}, and Duhamel's formula for $u$ to complete the proof.
\end{proof}

\da{
\begin{example}
Focusing on the case when $p=\I$ for simplicity, we include an example to demonstrate that the conditions stipulated  in the theorem are not overdetermined nor do they result in a smallness condition in $L^\I$.    Fix $\ga = 1/2$. We will show that for any $\be, \varepsilon\in (0,1)$ and $\ell>0$ there exists a divergence-free vector field $u_0$ such that $\|u_0\|_{L^\I}=1$ and $u_0$ is $(\e,\be,\ell)$-sparse.  Let $\phi :\R\to[0,1]$ satisfy $\phi = 1$ on $[-r,r]$ and $\phi = 0$ on $[-2r,2r]^c$ where $r>0$ will be chosen momentarily. Let $u_0$ be the vector field $(0,\phi(x_1)\phi(x_3),0)$. As a shear flow, this vector field is plainly divergence free. Furthermore, for any $r$, $\|u_0\|_{L^\I}=1$. (In any case, sparseness is preserved under multiplication by a non-zero prefactor.) By choosing $r$ we will ensure the appropriate sparseness condition is satisfied by setting $S = \{x\in \R^3: |x_1|<2r \mbox{ and }|x_3|<2r  \}$. In particular, we want $u_0(x)=0$ on $S^c$ and hence $\| u_0\|_{L^\I(S^c)}<\be \|u_0\|_{L^\I}$ for any $\be\in (0,1)$. Additionally, for any $\ell>0$ and $\varepsilon>0$ we may take $r$ small compared to $\ell$ so that 
\[
\sup_{x\in \R^d} \frac{|S\cap B_\ell(x)|} {|B_\ell(x)|} \leq  \varepsilon.
\]
Because we are free to choose any values for $\be,$ $\varepsilon$ and $\ell$, we may  choose them to be the values stated in the theorem. 
\end{example}
}

\begin{proof}[Proof of Theorem~\ref{thm:regcrit}] It suffices to show that, for a time $t'\in (0,\bar{T})$, we have 
\begin{equation}
    \label{eq:makemuhudrop}
\|u(\cdot,t')\|_{L^p} T_0^{\frac{1}{2} (1-\frac{d}{p})} \leq c_p
\end{equation}
since this guarantees that the solution is strong on the time interval $(0,t'+T_0)$. We  achieve~\eqref{eq:makemuhudrop} by applying Proposition~\ref{prop.decreasing} with $\gamma$ as above and setting \da{$t' = T_1$}, provided that $u_0$ is appropriately $(\varepsilon,\beta,\ell)$-sparse in $L^p$.
\end{proof}

\begin{proof}[Proof of Corollary~\ref{cor:blowupcrit}]
This follows immediately from Theorem~\ref{thm:regcrit} applied at time $t \in (0,T^*)$ with $u_0 = u(\cdot,t)$ and $T_0= T^*-t $.
\end{proof}
 

\begin{proof}[Proof of Theorem~\ref{thm:apriori}]
Let $b \in (0,1)$ and define
\begin{equation}
    \label{eq:Stdef}
S_t = \{ x: |u(x,t)| > b \ell_0^{-\frac{d}{p}} \| u(\cdot,t)\|_{L^p} \}.
\end{equation}
Chebyshev's inequality~\eqref{eq:chebyshev} with $f = |u|^2$ and $\lambda = b^2 \ell_0^{-\frac{2d}{p}} \| u(\cdot,t) \|_{L^p}^2$ gives
\begin{equation}
    | S_t | \leq \frac{\ell_0^{\frac{2d}{p}} \| u(\cdot,t) \|_{L^2}^2}{b^2 \| u(\cdot,t) \|_{L^p}^2} = \frac{\ell_0^{\frac{2d}{p} + \frac{2}{\mu_p}}}{b^2}  = \frac{\ell_0^d}{b^2}.
\end{equation}
As before, we consider balls with radius $\ell = a \ell_0$. Then the solution will be sparse at scale $\ell$ when $a \gg 1$:
\begin{equation}
|S_t \cap B_\ell(x_0) | \leq \frac{C_d}{a^d b^2} |B_\ell|,
\end{equation}
provided that we verify that $\| u(\cdot,t) \|_{L^p(S^c)} < \| u(\cdot,t) \|_{L^p}$ (when $p=+\infty$, this step is automatic). By interpolation and the definition of $S_t$, we have
\begin{equation}
\begin{aligned}
  \| u(\cdot,t)\|_{L^p(S^c)} &\leq \| u(\cdot,t)\|_{L^\I(S^c)}^{1-\frac{2}{p}} \| u(\cdot,t)\|_{L^2}^{\frac 2 p}
\\&\overset{\eqref{eq:Stdef}}{\leq} \left( b \ell_0^{-\frac{d}{p}} \| u(\cdot,t) \|_{L^p} \right)^{1-\frac{2}{p}} \|u(\cdot,t)\|_{L^2}^{\frac 2 p}
\\&= b^{1-\frac{2}{p}} \ell_0^{-\frac{d}{p}(1-\frac{2}{p})} \left( \frac{\| u(\cdot,t) \|_{L^2}}{\| u(\cdot,t) \|_{L^p}} \right)^{\frac{2}{p}} \| u(\cdot,t) \|_{L^p} \\
&\overset{\eqref{eq:ellnotdef}}{\leq} b^{1-\frac{2}{p}} \| u(\cdot,t) \|_{L^p}
\end{aligned}
\end{equation}
Choosing $b$ to satisfy $b^{1-\frac{2}{p}} = \beta$ and $a$ sufficiently large depending on $\varepsilon$ and $\beta$ completes the proof.
\end{proof}

\subsection{Frequency sparseness}
\label{sec:frequencysparseness}

In this section, we explore analogous theorems concerning frequency sparseness, see Definition~\ref{def:frequencysparseness}. 

\da{Our first result is an analogue of Proposition \ref{prop.decreasing}.}
\begin{proposition}\label{prop.decreasing.frequency}
Let $d < p\leq \infty$ and $u_0\in L^p_\sigma(\R^d)$. For any $\ga\in (0,1)$, \da{define $T_1\in (0,\bar T)$ according to \eqref{eq:Tdef}.} Then \da{$T_1$} satisfies the following properties: First, the strong $L^p$ solution $u$ of the Navier-Stokes equations with initial data $u_0$ exists on $\R^d \times (0,\da{T_1}]$. Moreover, for any $t\in (0,\da{T_1}]$ if $u_0$ is  $(\be,J)$-sparse where
\[
\be =\frac \ga 2  ;\quad 2^J \sim \ga^{-1}t^{-1/2},
\]
then
\[
\|  u(\cdot,t)\|_{L^p} \leq \ga \|u_0\|_{L^p}.
\]
\end{proposition} 

\begin{proof}
\da{From \eqref{ineq.Duhamel} and the definition of $T_1$,
for all $t \in (0,T_1]$, we have }
\begin{equation}
\bigg\|  \int_0^t e^{(t-s)\Delta}\mathbb P \nb \cdot (u\otimes u)\,ds\bigg\|_{L^p}
\leq  \frac \ga 2 \|u_0\|_{L^p}.
\end{equation}
Applying Lemma \ref{lemma.heat.frequency} ensures that, if $u_0$ is $(\ga/2, J)$-sparse where $J$ is defined by $2^J\sim \ga^{-1}t^{-1/2}$ then 
\begin{equation}
\|  e^{t\Delta}u_0 \|_{L^p}\leq \frac \ga 2 \|u_0\|_{L^p}.     
\end{equation}
This concludes the proof.
\end{proof}

Based on Proposition \ref{prop.decreasing.frequency}, we may revisit Theorems \ref{thm:regcrit} and \da{\ref{thm:apriori}, as well as Corollary \ref{cor:blowupcrit}}, from the frequency perspective.

\begin{corollary}[Frequency regularity criterion]\label{cor.zoran.sparseness.frequency}
Let $d < p\leq \I$, $u_0\in L^p_\sigma$, and \da{$T_0 \geq \bar{T}$}. Let $\gamma$ be  defined by
\begin{equation}
    \| u_0 \|_{L^p} T_0^{\frac{1}{2}(1-\frac{d}{p})} \gamma = c_p.
\end{equation}
\da{Define $T_1$ according to \eqref{eq:Tdef}.}
Assume that $u_0$ is $(\be,J)$-sparse in frequency in $L^p$, where 
\begin{equation}
\be = \frac{\gamma}{2}; \quad 2^J \sim \gamma^{-1} \da{T_1^{-1/2}}
\end{equation}
and the suppressed constants are dimensionless and universal.
Then, the unique mild solution for $u_0$ exists and remains smooth beyond time $T_0$.
\end{corollary}

\begin{proof} 
 It suffices to show that, for a time $t'\in (0,\bar{T})$, we have 
\begin{equation}
    \label{eq:makemuhudrop2}
\|u(\cdot,t')\|_{L^p} T_0^{\frac{1}{2} (1-\frac{d}{p})} \leq c_p
\end{equation}
since this guarantees that the solution is strong on the time interval $(0,t'+T_0)$. We can achieve~\eqref{eq:makemuhudrop2} by applying Proposition~\ref{prop.decreasing.frequency} with $\gamma$ as above,
taking \da{$t=T_1$} and requiring $u_0$ to be appropriately $(\be,J)$-sparse.\end{proof}

\da{
In the next corollary, we define $T_1(t)$ as in Corollary \ref{cor:blowupcrit}.}

\begin{corollary}[Frequency blow-up criterion]
Let $d < p\leq \I$ and $u_0\in L^p_\sigma$. Let $u$ be the unique strong solution for $u_0$. Suppose that the maximal time of existence $T^*$ is finite. 
Let $\gamma_p(t)$ be defined by  
\begin{equation}
     \| u(\cdot,t) \|_{L^p} (T^* - t)^{\frac{1}{2} (1-\frac{d}{p})} \gamma_p(t) = c_p.
\end{equation}
Let $J(t)$ be defined by the property that
\[
\| u_{<J}(\cdot,t)\|_{L^p} < \frac 1 2 \|u(\cdot,t)\|_{L^p}.
\]
Then
\[
2^{J(t)} {\les} \gamma^{-1} \da{T_1(t)}^{-\frac{1}{2}}
\]
In the case of Type-1 blow-up, i.e.,~$\| u(\cdot,t) \|_{L^p} \sim (T^*-t)^{-\frac{1}{2}(1-\frac{d}{p})}$, we have 
\[
2^{J(t)} {\les}  (T^*-t)^{{-\frac{1}{2}}}.
\]
 \end{corollary}
 \begin{proof}
 This follows immediately from Corollary~\ref{cor.zoran.sparseness.frequency} initiated at time $t \in (0,T^*)$ with $u_0 = u(\cdot,t)$ and $T_0= T^*-t$.
 \end{proof}


\begin{proposition}[\textit{A priori} frequency sparseness]Let $d<p\leq \infty$. Suppose that $u$ is a strong Navier-Stokes solution with initial data $u_0\in L^2_\sigma \cap L^p_\sigma$. Let $\ell_0$ and $\mu_p$ be defined as in Theorem \ref{thm:apriori}. Then, for any $\be\in (0,1)$, letting $2^{-J} \sim \ell_0 \da{\beta^{-\mu_p}}$, we have $u(\cdot,t)$ is $(\be, J)$-sparse in frequency.
\end{proposition}

\begin{proof}
By Bernstein's inequality, we have,
\begin{equation}
    \| \Delta_{\leq J} u(\cdot,t) \|_{L^p} \lesssim \frac {2^{J(\frac d 2 - \frac d p)} \| u_0 \|_2 } { \| u(\cdot,t)\|_{L^p}} \| u(\cdot,t)\|_{L^p}.
\end{equation}
The proposition follows readily.
\end{proof}

We now use the heat kernel estimates to give \da{yet another proof}  (after \da{Cheskidov and Shvydkoy~\cite{Cheskidov2009}}, Farhat, Gruji\'c, and Leitmeyer~\cite{FarhatB1infty,FarhatB1inftyerr}, and Hmidi and Li~\cite{HmidiLi}) that if a strong solution is small in $\dot B^{-1}_{\I,\I}$ on any interval of time then it can be smoothly extended to a larger interval of time. \da{This illustrates a positive connection between sparseness and regularity.}

\begin{proposition}\label{prop:smallBesov}Let $u_0 \in L^\I_\sigma$. Let $u$ be the unique strong solution with initial data $u_0$. Let $T^* \in (0,+\infty]$ be the maximal time of existence. There exists a universal constant $\e_*$ such that if 
\begin{equation}
\sup_{0<t<T^*}\| u(\cdot,t)\|_{\dot B^{-1}_{\I,\I}} <\e_*,    
\end{equation}
then $T^* = +\infty$.
\end{proposition}
 
\begin{proof}
Assume that $T^* < +\infty$.
Note that $u\in C((0,T);L^\I)$. Since $\| u(\cdot,t)\|_{L^\I}$ blows up as $t\to T^*_-$,  for any $M>\| u_0\|_{L^\I}$, there exists $t_M$ so that $\|u(\cdot,t_M)\|_{L^\I} = M$ and $\sup_{t_M<t<T^*}\| u(\cdot,t)\|_{L^\I} >M$. The times $t_M$ are called \emph{escape times}. We will arrive at a contradiction by proving no escape times exist. Note that, for all $t$ sufficiently close to $T^*$, 
\begin{equation}
    \label{lower.bound}
\| u(\cdot,t)\|_{L^\I} \geq c_\infty (T^*-t)^{-\frac 12}.
\end{equation}
Let $\ga = 1/2$ and $2^{J} = 2 (t'-t)^{-1/2}$ where $t' = t+ (T^*-t)/2 = T^*/2 +t/2$.  We will show that if $\|u\|_{L^\I_{t} \dot B^{-1}_{\I,\I}}$ is small then $u(\cdot,t)$ is $(\ga, J)$ sparse and, by Proposition~\ref{prop.decreasing.frequency}, $\|u(\cdot,t')\|_{L^\I} \leq \| u(\cdot,t)\|_{L^\I}$. This implies $t$ is not an escape time. Since $t$ was arbitrary, there exist no escape times close to $T^*$ and, therefore, $T^*$ is not a blow-up time.

We have by the definitions of the norm of $\dot B^{-1}_{\I,\I}$ and   $t'$ that 
\begin{equation}
\begin{aligned}
   \| \Delta_{\leq J} u(\cdot,t)\|_{L^\I} &\leq C2^J \| u (\cdot,t)\|_{\dot B^{-1}_{\I,\I}} \\
   &= C \bigg(\frac 1 2 (T^*-t) \bigg)^{-\frac 1 2} \| u (\cdot,t)\|_{\dot B^{-1}_{\I,\I}}
\\&\lesssim  \|u(\cdot,t)\|_{L^\I}\| u (\cdot,t)\|_{\dot B^{-1}_{\I,\I}}, 
\end{aligned}
\end{equation}
where we  assume $t$ is close enough to $T^*$ for \eqref{lower.bound} to hold. Requiring
\begin{equation}
\| u \|_{L^\I_{t}\dot B^{-1}_{\I,\I}} \ll \be
\end{equation}
implies $u(\cdot,t)$ is $(\be,J)$-sparse.   
\end{proof}

\section{Examples}
\label{sec:examples}

\da{In this section, we examine the sparsity classes $Z^{(k)}_{\alpha}$ originally introduced in~\cite{grujic2019asymptotic} through the lens of concrete examples. The ansatz we consider is motivated by known singularity formation for related nonlinear PDEs, e.g., semilinear heat equations (see~\cite{Collotsurvey} and the references therein),  the three-dimensional compressible Navier-Stokes equations~\cite{merle2019implosion}, and Burgers equation with fractional dissipation~\cite{oh2021gradient,chickering2021asymptotically}, among others. }


For $x \in \R^d$ and $t \in (-T,0)$, we define the (backward) similarity variables
\begin{equation}
    \label{eq:similarityvars1}
	y = \frac{x}{(-t)^{\zeta_x}}, \quad s = -\log (-t)
\end{equation}
and
\begin{equation}
    \label{eq:theexampleansatz}
	u(x,t) = \frac{1}{(-t)^{\zeta_t}} U( y, s) \, ,
\end{equation}
where \zb{$u:\R^d\times (-T,0)\to \R^d$ is a vector field and} $\zeta_x,\zeta_t > 0$ are two positive exponents. \da{Let $S = - \log T$.} 
\zb{The vector field $U  : \R^d \times (S,+\infty) \to \R^d$ can be viewed as a `similarity profile.'}
\zb{Notably,   $U$ can remain controlled for $s>S$ and still correspond to blow-up in $u$ as $t\to 0^-$.} \zb{For every $k\in \N_0$,} we assume that $U$ satisfies 
\begin{equation}
    \da{ \sup_{s \in (S,+\infty)} } \|  \nabla_y^k U(\cdot,s) \|_{L^\infty} \leq C_k
\end{equation}
for some $C_k > 0$,
\begin{equation}
    \label{eq:decayofsolution}
    \da{ \sup_{s \in (S,+\infty)}  } |\nabla^k_y U| \; \da{ \to 0 \text{ as } |y| \to +\infty} \, , 
\end{equation}
and there exists $R_k > 0$ such that
\begin{equation}
    \label{eq:similarityvars5}
    \da{ \inf_{s \in (S,+\infty)} } \| \nabla_y^k U \|_{L^\infty(B_{R_k})} \geq c_k
\end{equation}
for some $c_k > 0$.   

\da{If $u$ satisfies the Navier-Stokes equations
\begin{equation}
    \label{eq:nse}
    \p_t u + u \cdot \nabla u - \Delta u + \nabla p = 0 \, , \quad \div u = 0 \, ,
\end{equation}
then $U$ satisfies the system 
\begin{equation}
\begin{aligned}
    \label{eq:similaritynavierstokes}
    &\p_s U + \zeta_t U + \zeta_x y \cdot \nabla_y U + e^{-(1-\zeta_t-\zeta_x) s} U \cdot \nabla_y U - e^{-(1-2\zeta_x)s} \Delta U + \nabla P = 0 \\
    &\div U = 0 \, .
    \end{aligned}
\end{equation}

The canonical choice of exponents in~\eqref{eq:similaritynavierstokes} is $\zeta_t = \zeta_x = 1/2$, which causes the exponential prefactors to be autonomous. With this choice, if moreover $U$ were steady, then $u$ would be a backward self-similar solution, as proposed in Leray's work~\cite[(3.11), p. 225]{leray}. Such solutions were excluded by~\cite{necasruzsverak,tsai,tsaierratum}. If instead $U$ were time-periodic, then $u$ would be a backward discretely self-similar solution; this scenario has not been ruled out.

A different but natural choice is $\zeta_t + \zeta_x = 1$ and $0 < \zeta_x < 1/2$, which causes the exponential prefactor in front of the term $U \cdot \nabla_y U$ to become autonomous, whereas the prefactor in front of $\Delta U$ converges to zero as $s \to +\infty$. In this scenario, the Navier-Stokes singularity would be a perturbation of an Euler singularity. \cmz{We emphasize that the above scenario is purely speculative, and the currently known Euler singularities~\cite{elgindiblowup} (see also the numerical predictions in~\cite{luo2014potentially}) cannot, to the best of our knowledge, be readily perturbed to Navier-Stokes singularities in this fashion.} Analogous blow-ups, in which the dissipation is perturbative, are known to occur for the three-dimensional compressible Navier-Stokes equations~\cite{merle2019implosion} and the Burgers equation with fractional dissipation~\cite{oh2021gradient,chickering2021asymptotically}.}  

\da{The ansatz~\eqref{eq:similarityvars1}-\eqref{eq:similarityvars5} may be thought of as describing a `shape' of blow-up, and it encompasses a wide range of possible behaviors, as it is not known  how a hypothetical Navier-Stokes singularity might realistically look. In particular, the decay requirement~\eqref{eq:decayofsolution} is general and might also capture functions beyond the Leray-Hopf class. Typically the `correct' behavior as $|y| \to +\infty$ may depend on the choice of exponents. For example, it does not make sense to generally impose that $U \in L^\infty_s L^2_y(\R^d \times (S,+\infty))$, as this would cause the blow-up profile $u(\cdot,0)$ to vanish identically when $\zeta_t = \zeta_x = 1/2$, thereby violating backward uniqueness~\cite{escauriazasereginsverak}. One can also imagine an inner blow-up region `glued' to a smooth outer region, as can be made rigorous for the harmonic map heat flow~\cite{weiharmonicmap}.

Our main observation is the following:
} 






\da{
\begin{proposition}
Let $d \geq 2$, $T, \zeta_t, \zeta_x > 0$, and $u \: \R^d \times (-T,0) \to \R^d$ be a vector field satisfying the hypotheses~\eqref{eq:similarityvars1}-\eqref{eq:similarityvars5}.  Then the following are equivalent:
\begin{enumerate}
    \item (Locally finite kinetic energy). $\sup_{t \in (-T,0)} \| u(\cdot,t) \|_{L^2(B_{(-t)^{\zeta_x}R_0 })} < +\infty$.
    \item (Sparsity). $u$ belongs to $Z^{(k)}_{\bar{\alpha}_k}$ uniformly in $t \in (-T,0)$. 
    \item (Exponents). $\zeta_t \leq d \zeta_x/2$.
\end{enumerate}
\end{proposition}
}

\da{In particular, within the class of examples we present here,  membership in $Z^{(k)}_{\bar{\alpha}_k}$ is not stronger than the condition of finite kinetic energy $u \in L^\infty_t L^2_x(\R^d \times (-T,0))$.}

\begin{proof}
 The \da{criteria~\eqref{eq:similarityvars1}-\eqref{eq:similarityvars5}} ensure two elementary sparseness properties on the profiles $U(\cdot,t)$:
 
 \da{\textbf{I}.} For any $\varepsilon, \beta \in (0,1)$, there exists $L_k \gg 1$ such that $\nabla_y^k U$ is sparse at scale $L_k$, uniformly in $s$. In this section, `sparse' means in the $L^\infty$ sense. Indeed, we define 
 $S = \{ |\nabla^k U| > \beta \| \nabla^k U \|_{L^\infty} \}$. Then
 \begin{equation}
     S \subset \{ |\nabla^k U| > \beta c_k \} \subset B_{\bar{R}_k}
 \end{equation}
 for sufficiently large $\bar{R}_k$. Hence,
 \begin{equation}
     |S \cap B_\ell(x_0)| \leq |B_{\bar{R}_k}| \leq \left( \frac{\bar{R}_k}{\ell} \right)^d |B_\ell|,
 \end{equation}
 and we choose $L_k = \ell \gg \bar{R}_k$.
 
 \da{\textbf{II}.} For any $\varepsilon,\beta \in (0,1)$, there exists $\ell_k \ll 1$ such that the components of $\nabla_y^k U$ are \emph{not} sparse at scale $\ell_k$. Indeed, choose $x_k$ (depending on $t$) such that $|\nabla^k U(x_k)| = \| \nabla^k U \|_{L^\infty}$. Since $\| \nabla^{k+1} U \|_{L^\infty} \leq C_{k+1}$, we have that for $\ell \ll_\beta \| \nabla^k U \|_{L^\infty} / C_{k+1}$,
 \begin{equation}
     |\nabla^k U(x_k)| > \beta \| \nabla^k U \|_{L^\infty} \text{ in } B_\ell(x_k).
 \end{equation}

We now verify the scales of sparseness in the above examples. Clearly, we have
\begin{equation}
	\nabla_x^k u = \frac{1}{(-t)^{\zeta_t+ k \zeta_x}} (\nabla_y^k U)(y,s).
\end{equation}
Therefore, \da{according to \textbf{I}}, the solution $u$ and its derivatives $\nabla^k_x u$ are sparse at scale $\ell = (-t)^{\zeta_x} L_k$. The norm of $\nabla^k u$ in $L^\infty$ is $\sim_k (-t)^{-(\zeta_t+ k \zeta_x)}$, bounded above and below up to multiplicative constants. Multiplying this exponent by \da{a number} $-\alpha_k$ \da{(depending on $\zeta_t$ and $\zeta_x$)} and matching it to the exponent $\zeta_x$, we discover that $u(\cdot,t) \in Z^{(k)}_{\alpha_k}$, uniformly in $t$, with
\begin{equation}
    \label{eq:alpharequirements}
	\alpha_k = \frac{\zeta_x}{\zeta_t + k \zeta_x} = \frac{1}{\frac{\zeta_t}{\zeta_x} + k}.
\end{equation}
\da{Hence,
\begin{equation}
    \label{eq:uisthere}
    u \text{ belongs to } Z^{(k)}_{\alpha_k} \text{ uniformly in } t \in (-T,0) \, .
\end{equation}
}
Moreover, \da{according to \textbf{II}}, we have that
\begin{equation}
    \label{eq:unotthere}
    u \text{ fails to belong to } Z^{(k)}_{\alpha} \text{ uniformly in } t \in (-T,0) \text{ for any } \alpha > \alpha_k \, .
\end{equation}

We now compute what exponents are admissible \da{under various conditions}. We begin by computing the norms
\begin{equation}
    \label{eq:inspecthexponents}
    \| \nabla^k_x u(\cdot,t) \|_{L^p\left( B_{(-t)^{\zeta_x} R_k} \right)} \sim_{p,k} (-t)^{- \zeta_t + \left( \frac{d}{p} - k\right) \zeta_x }.
\end{equation}

1. \emph{\da{Locally finite} kinetic energy}: By inspecting~\eqref{eq:inspecthexponents} with $p=2$ and $k= 0$, we observe that the exponents which keep~\eqref{eq:inspecthexponents} bounded \da{are precisely those satisfying}
\begin{equation}
    \label{eq:finitekineticenergy}
    \zeta_t  \leq \frac{d}{2}\zeta_x \, .
\end{equation}

2. \emph{The class $Z^{(k)}_{\bar{\alpha}_k}$}: Recall that
\begin{equation}
    \label{eq:baralphakdef}
    \bar{\alpha}_k = \frac{1}{k+\frac{d}{2}}
\end{equation}
is the exponent identified in~\cite{grujic2019asymptotic}.
By comparing~\eqref{eq:alpharequirements} with~\eqref{eq:baralphakdef}, we discover that, \da{due to~\eqref{eq:uisthere} and~\eqref{eq:unotthere}}, the condition \da{that $u$ belongs to $Z^{(k)}_{\bar{\alpha}_k}$ uniformly in $t \in (-T,0)$ is equivalent to $\alpha_k \geq \bar{\alpha}_k$, that is,}
\begin{equation}
    \label{eq:baralphareqs}
    \zeta_t \leq \frac{d}{2} \zeta_x,
\end{equation}
which is the same condition imposed by \da{locally} finite kinetic energy. \da{To conclude}, within the class of examples we present here, having \da{locally} finite kinetic energy and membership in $Z^{(k)}_{\bar{\alpha}_k}$ occur for exactly the same exponents.
\end{proof}

While the above observation is already our main point, it is interesting to discuss further requirements on the exponents:

3. \emph{Energy class}. If $\nabla_x u \in L^2_t L^2_x(B_{(-t)^{\zeta_x} R_k} \times (-T,0))$, then
\begin{equation}
    \zeta_t < \left( \frac{d}{2} - 1 \right) \zeta_x + \frac{1}{2}.
\end{equation}
That is, the Leray--Hopf class excludes more exponents than merely having \da{locally} finite kinetic energy does.

4. \emph{Known regularity criteria}. For $u$ to be a singular solution, it must satisfy $\| u(\cdot,t) \|_{L^\infty} > \da{c_\infty} (-t)^{-1/2}$, from which we obtain $\zeta_t \geq 1/2$. Similarly, if \emph{additionally} we assume \da{the decay condition} $U \in L^\infty_t L^{d,\infty}_y$, then the requirement $\liminf_{t \to 0_-} \| u(\cdot,t) \|_{L^{d,\infty}} > 0$ implies that \da{$\zeta_t \geq \zeta_x$.}  We assume the additional spatial decay requirement on the profile because~\eqref{eq:inspecthexponents} only describes the norm in a ball whose radius is shrinking to zero. The most reasonable spatial decay conditions on $U$ \da{depend on $\zeta_t$ and $\zeta_x$ and} are those which cause the `blow-up profile' $u(\cdot,0)$ to be well defined and non-trivial. \da{For example, one might instead require $|U| \sim |y|^{-\zeta_t/\zeta_x}$ as $|y| \to +\infty$, which yields, in particular, $u \in L^\infty_t L_x^{d\zeta_x/\zeta_t,\infty}$.}  

\begin{figure}[h!]
		\centering
		\includegraphics[width=0.8\linewidth]{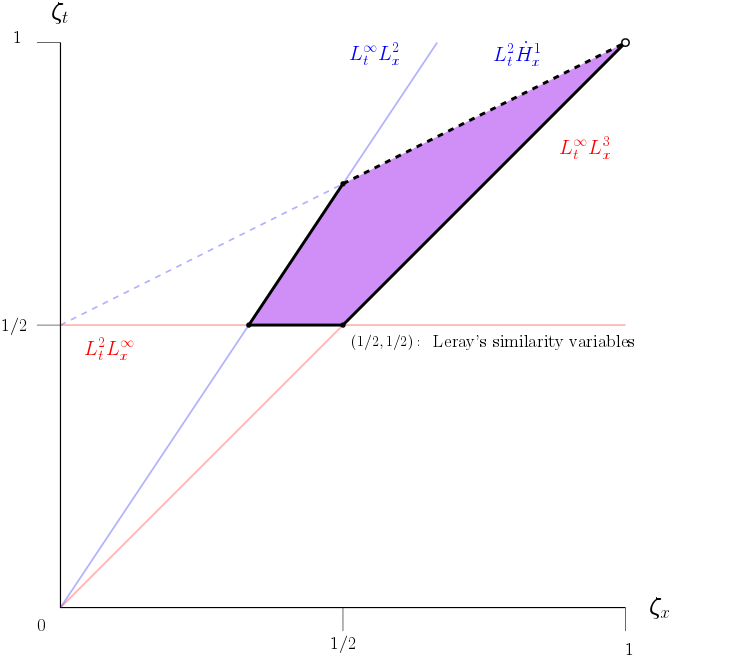}
	\caption{Admissible exponents $(\zeta_x,\zeta_t)$ for a hypothetical singular solution $u$ satisfying the ansatz~\eqref{eq:theexampleansatz}. The purple region denotes those exponents that are consistent with finite energy and do not satisfy the known (critical) regularity criteria.
	\zb{Exponents for $Z^{(k)}_{\bar{\alpha}_k}$ are found on the line corresponding to $L^\I_t L^2_x$. A reduction of the scaling gap would correspond to the sparseness class exponents belonging } \zb{to the interior of the purple region.}
		}
\label{fig:exponents}
\end{figure}

The ansatz~\eqref{eq:theexampleansatz} is by no means comprehensive. \da{For instance, it does not accommodate possible logarithmic corrections to the scaling laws, as are known to arise in the two-dimensional harmonic map heat flow~\cite{weiharmonicmap}, where the proper spatial rescaling is $y = |\log (-t)|^2 x/(-t)$.} \da{It is an interesting and important question to understand whether further analysis of the PDE~\eqref{eq:similaritynavierstokes} may disqualify further exponents.}




\subsubsection*{Acknowledgments}
DA was supported by NSF Postdoctoral Fellowship  Grant No. 2002023. ZB was supported in part by the Simons Foundation via Collaboration Grant No. 635438. The authors thank Zoran Gruji{\'c} and Theo Drivas for feedback on a preliminary version and Tobias Barker and the anonymous referees for pointing out an error in the original manuscript and other valuable comments.

\bibliography{bibliography}
\bibliographystyle{alpha}

\parskip    0mm
\end{document}